\documentclass{amsart}
\usepackage{amssymb,amscd,amsthm,latexsym}
\usepackage[all]{xy}
\CompileMatrices

\newtheorem{theorem}{Theorem}[section]
\newtheorem{lemma}[theorem]{Lemma}
\newtheorem{proposition}[theorem]{Proposition}
\newtheorem{corollary}[theorem]{Corollary}
\newtheorem*{property-P}{Property P}

\theoremstyle{definition}

\newtheorem{definition}[theorem]{Definition}

\renewcommand{\to}{\longrightarrow}

\newcommand{\Ac}{\ensuremath{\mathcal{A}}}

\newcommand{\Top}{\ensuremath{\mathsf{Top}}}

\newcommand{\Xc}{\ensuremath{\mathcal{X}}}

\newcommand{\RG}{\ensuremath{\mathsf{RG}}}

\newcommand{\Set}{\ensuremath{\mathsf{Set}}}

\newcommand{\cod}{\ensuremath{\mathrm{cod\,}}}
\newcommand{\dom}{\ensuremath{\mathrm{dom}}}

\newcommand{\Ext}{\ensuremath{\mathsf{Ext}}}

\newcommand{\Reg}{\ensuremath{\mathsf{Reg}}}
\newcommand{\Mon}{\ensuremath{\mathsf{Mon}}}
\newcommand{\Pt}{\ensuremath{\mathsf{Pt}}}

\newcommand{\DiscFib}{\ensuremath{\mathsf{DiscFib}}}

\newbox\skewpullbackbox
\setbox\skewpullbackbox=\hbox{\xy 0;<1mm,0mm>: \POS(4,0)\ar@{-} (-4,0) \ar@{-} (8,4)
\endxy}

\newbox\ksewpullbackbox
\setbox\ksewpullbackbox=\hbox{\xy 0;<1mm,0mm>: \POS(0,-8)\ar@{-} (0,-4) \ar@{-} (4,-4)
\endxy}

\newbox\pullbackbox
\setbox\pullbackbox=\hbox{\xy 0;<1mm,0mm>: \POS(4,0)\ar@{-} (0,0) \ar@{-} (4,4)
\endxy}
\def\pullback{\copy\pullbackbox}

\newbox\pushoutbox
\setbox\pushoutbox=\hbox{\xy 0;<1mm,0mm>: \POS(0,4)\ar@{-} (0,0) \ar@{-} (4,4)
\endxy}

\begin{document}

\title{Effective descent morphisms of regular epimorphisms}

\author{Tomas Everaert}
  \address{
Vakgroep Wiskunde \\
Vrije Universiteit Brussel \\
 Department of Mathematics  \\
 Pleinlaan 2\\
1050 Brussel \\
 Belgium.
 }

\email{teveraer@vub.ac.be}

 \date{\today}

\maketitle

\begin{abstract}
Let $\Ac$ be a regular category with pushouts of regular epimorphisms by regular epimorphism and $\Reg(\Ac)$ the category of regular epimorphisms in $\Ac$. We prove that every regular epimorphism in $\Reg(\Ac)$ is an effective descent morphism if, and only if, $\Reg(\Ac)$ is a regular category. Then, moreover, every regular epimorphism in $\Ac$ is an effective descent morphism. This is the case, for instance, when $\Ac$ is either exact Goursat, or ideal determined, or is a category of topological Mal'tsev algebras, or is the category of $n$-fold regular epimorphisms in any of the three previous cases, for any $n\geq 1$.
\end{abstract}
\pagestyle{myheadings}{\markboth{TOMAS EVERAERT},\markright{}}

\section{Introduction}

A useful way of weakening the notion of Barr exactness, for a regular category, is to require that every regular epimorphism is an effective descent morphism, which assures the effectiveness of a certain class of equivalence relations, rather than of all. This weaker condition turns out to be strong enough for many purposes: indeed, the authors of \cite{Janelidze-Sobral-Goursat} had good reasons to say that a regular category satisfying this condition is ``almost Barr exact''!

In the present article, we are interested in the category $\Reg(\Ac)$ of regular epimorphisms in a regular category $\Ac$. This category is usually not exact. For instance, if $\Ac$ is a non-trivial abelian category, then $\Reg(\Ac)$ is (equivalent to) the category of short exact sequences in $\Ac$, which is well known not to be abelian, while it is obviously additive---hence it can not be exact by ``Tierney's equation''. However, there are many examples of categories $\Ac$ for which every regular epimorphism in $\Reg(\Ac)$ is an effective descent morphism.  Indeed, this is the case, for instance,  for any abelian category and, much more generally, for any exact Goursat category $\Ac$, as was shown in  \cite{Janelidze-Sobral-Goursat}. It was pointed out in \cite{Janelidze-Sobral-Goursat}, however, that the exact Goursat condition is most likely too strong, even if $\Ac$ is assumed to be a variety. This is confirmed in the present article.  In particular, we show, for a regular category $\Ac$ with pushouts of regular epimorphisms by regular epimorphisms that, in order to have that every regular epimorphism in $\Reg(\Ac)$ is an effective descent morphism, it is both necessary and sufficient  that also $\Reg(\Ac)$ is a regular category. This condition turns out to be satisfied not only in the exact Goursat case, but also when $\Ac$ is ideal determined, or is a category of topological Mal'tsev algebras, or is the category of $n$-fold regular epimorphisms in any of the three previous cases, for any $n\geq 1$. In particular, we find that in each of these cases the category $\Ac$ is itself ``almost exact''.

Let us begin by recalling the definition of an effective descent morphism---for instance from \cite{Janelidze-Sobral-Tholen}, but note that the monadic description recalled here goes back to B\'enabou and Roubaud's article \cite{Benabou-Roubaud}; for the reformulation in terms of discrete fibrations, see also \cite{Janelidze-Tholen2}. 

Let $\Ac$ be a category with pullbacks. If $B$ is an object of $\Ac$, then we write $(\Ac\downarrow B)$ for the slice category over $B$. 
If $p\colon E\to B$ is a morphism in $\Ac$, we write $p^*\colon (\Ac\downarrow B)\to (\Ac\downarrow E)$ for the induced ``change of base'' functor given by pulling back along $p$.

\begin{definition}
An effective (global) descent morphism in a category with pullbacks $\Ac$ is a morphism $p\colon E\to B$ such that $p^*\colon (\Ac\downarrow B)\to (\Ac\downarrow E)$ is monadic.
\end{definition}
Note that a left adjoint for $p^*\colon (\Ac\downarrow B)\to (\Ac\downarrow E)$ exists for any morphism $p\colon E\to B$, and is given by composition with $p$. We denote it $\Sigma_p$, and write $T^p=p^*\circ \Sigma_p$ for the corresponding monad on $(\Ac\downarrow E)$. Writing $(\Ac\downarrow E)^{T^p}$ for the corresponding category of (Eilenberg-Moore) algebras, we obtain the following  commutative triangle of functors, where $U^{T^p}$ is the forgetful functor and $K^{T^p}$ the comparison functor:
\[
\xymatrix{
(\Ac\downarrow B) \ar[d]_{K^{T^p}} \ar[r]^{p^*} & (\Ac\downarrow E)\\
(\Ac\downarrow E)^{T^p} \ar[ru]_{U^{T^p}} &}
\]
Thus $p\colon E\to B$ is an effective descent morphism if and only if $K^{T^p}$ is a category equivalence. Note that when $K^{T^p}$ is merely full and faithful, one says that $p$ is a \emph{descent morphism}.

There is an equivalent way of describing the above diagram, via discrete fibrations.  Recall that a discrete fibration of equivalence relations in $\Ac$ is a (downward) morphism
\[
\xymatrix{
S \ar@{}[rd]|<<<{\pullback} \ar@<0.5 ex>[r] \ar@<-0.5 ex>[r] \ar[d] & A \ar[d]\\
R \ar@<0.5 ex>[r] \ar@<-0.5 ex>[r] & E }
\]
of equivalence relations, such that the commutative square involving the second projections (hence, also the square involving the first projections) is a pullback. Let $p\colon E\to B$ be a morphism in $\Ac$. Write $Eq(p)$ for the equivalence relation $\xymatrix{E\times_B E\ar@<0.5 ex>[r] \ar@<-0.5 ex>[r] & E}$ (i.e. the kernel pair of $p$) and $\DiscFib (Eq(p))$ for the category of discrete fibrations over $Eq(p)$, with the obvious morphisms. It was proved in \cite{Janelidze-Tholen2} (but see also \cite{Janelidze-Sobral-Tholen}) that for any $p\colon E\to B$ the category of algebras $(\Ac\downarrow E)^{T^p}$ is equivalent to the category $\DiscFib (Eq(p))$ of discrete fibrations over the kernel pair of $p$, and the commutative diagram above becomes:  
\[
\xymatrix{
(\Ac\downarrow B) \ar[d]_{K^{p}} \ar[r]^{p^*} & (\Ac\downarrow E)\\
\DiscFib (Eq(p)) \ar[ru]_{U^{p}} &}
\]
Here $K^p$ sends a morphism $f\colon A\to B$ to the discrete fibration displayed in the left hand side of the diagram below, obtained by first pulling back $f$ along $p$, and next taking kernel pairs:
\[
\xymatrix{
E\times_B E\times_B A \ar@{}[rd]|<<{\pullback} \ar[d] \ar@<0.5 ex>[r] \ar@<-0.5 ex>[r] & E\times_BA \ar@{}[rd]|<<{\pullback} \ar[r] \ar[d] & A \ar[d]^f\\
E\times_B E \ar@<0.5 ex>[r] \ar@<-0.5 ex>[r] & E \ar[r]_p & B}
\]
and $U^p$ is the obvious forgetful functor.  

We shall need the following lemma, which can be found, for instance, in \cite{Janelidze-Sobral-Tholen}.  Recall that an equivalence relation in a category is \emph{effective} if it is the kernel pair of some morphism. A category is \emph{regular} if it is finitely complete, with pullback-stable regular epimorphisms and coequalisers of effective equivalence relations. It is \emph{Barr exact} if, moreover, every internal equivalence relation is effective.

\begin{lemma}\begin{enumerate}\label{known}
\item
In a finitely complete category, a descent morphism is the same as a pullback-stable regular epimorphism.
\item
In a regular category, a regular epimorphism $p\colon E\to B$ is an effective descent morphism if and only if for any discrete fibration 
\[
\xymatrix{
R  \ar@{}[rd]|<<<{\pullback} \ar[d] \ar@<0.5 ex>[r]^{\pi_1} \ar@<-0.5 ex>[r]_{\pi_2} & A \ar[d] \\
E\times_B E \ar@<0.5 ex>[r] \ar@<-0.5 ex>[r] & E}
\]
over the kernel pair of $p$, the equivalence relation $(\pi_1, \pi_2)$ is effective.
\end{enumerate}
\end{lemma}
Note that the second part of this lemma immediately implies the well-known fact that in an exact category every regular epimorphism is an effective descent morphism.

\section{Main results}
Throughout this section, we shall assume that $\Ac$ is a regular category with pushouts of regular epimorphisms by regular epimorphisms. Then, in particular, every morphism in $\Ac$ factors (essentially uniquely) as a regular epimorphism followed by a monomorphism.  Let us denote by $\Reg(\Ac)$ the full subcategory of the arrow category $\Ac^2$ with as objects all regular epimorphisms in $\Ac$. Thus, a morphism $(a\colon A'\to A)\to (b\colon B'\to B)$ is a pair $(f',f)$ of morphisms in $\Ac$ such that $b\circ f'=f\circ a$. Like $\Ac$, the category $\Reg(\Ac)$ is finitely complete:  limits are given by the regular epi part of the regular epi-mono factorisation of the degreewise pullback. An effective equivalence relation in $\Reg(\Ac)$ is the same as a graph
\[
\xymatrix{
R' \ar@<0.5 ex>[r]^{\pi_1'}  \ar@<-0.5 ex>[r]_{\pi_2'} \ar[d]_r & E' \ar[d]^e\\
R \ar@<0.5 ex>[r]^{\pi_1}  \ar@<-0.5 ex>[r]_{\pi_2}  & E}
\]
in $\Reg(\Ac)$ such that $(\pi_1',\pi_2')$ is an effective equivalence relation in $\Ac$ and $\pi_1$ and $\pi_2$ are jointly monic. A regular epimorphism in $\Reg(\Ac)$ is the same as a pushout square of regular epimorphisms in $\Ac$. Notice that a regular category $\Ac$ admits pushouts of regular epimorphisms by regular epimorphisms if and only if $\Reg(\Ac)$ admits coequalisers of effective equivalence relations.

By Lemma \ref{known}.1, any effective descent morphism in $\Reg(\Ac)$ is necessarily a regular epimorphism. We would like to know when we have the converse. Certainly, this can only happen if every regular epimorphism \emph{in $\Ac$} is an effective descent morphism, since for any regular epimorphism $p
\colon E\to B$ in $\Ac$ a category equivalence $K^{(p,p)}\colon (\Ac^2\downarrow 1_B)\to \DiscFib(Eq(p,p))$  will restrict to an equivalence $K^p\colon (\Ac\downarrow B)\to \DiscFib(Eq(p))$. Moreover, by Lemma \ref{known}.1,   $\Reg(\Ac)$ will need to have pullback-stable regular epimorphisms, which means---since it is finitely complete and has coequalisers of effective equivalence relations---that $\Reg(\Ac)$ is regular.

Somewhat surprisingly, perhaps, these conditions turn out to be sufficient. Indeed, we have: 

\begin{theorem}\label{maintheorem}
For a regular category $\Ac$ with pushouts of regular epimorphisms by regular epimorphisms, in which every regular epimorphism is an effective descent morphism, the following conditions are equivalent:
\begin{enumerate}
\item
Every regular epimorphism in $\Reg(\Ac)$ is an effective descent morphism;
\item
$\Reg(\Ac)$ is a regular category.
\end{enumerate}
\end{theorem}
\begin{proof}
We only need to prove that 2. implies 1. so let us assume that $\Reg(\Ac)$ is regular. To prove that condition 1. holds, it suffices, by Lemma \ref{known}.2, to show for every discrete fibration of equivalence relations in $\Reg(\Ac)$, as in the diagram
\begin{equation}\label{dfcube}\vcenter{
\xymatrix{
& S \ar@<0.5 ex>[rr]^{\overline{\pi}_1} \ar@<-0.5 ex>[rr]_{\overline{\pi}_2} \ar@{.>}[dd]_>>>>>g && A \ar[dd]^f \\
S'  \ar@<0.5 ex>[rr]^>>>>>>>>>{\overline{\pi}_1'} \ar@<-0.5 ex>[rr]_>>>>>>>>>{\overline{\pi}_2'}  \ar[dd] \ar[ru] && A' \ar[dd] \ar[ru]\\
& R \ar@<0.5 ex>@{.>}[rr]^<<<<<<<<<{\pi_1}  \ar@<-0.5 ex>@{.>}[rr]_<<<<<<<<<{\pi_2}  && E\\
R' \ar@<0.5 ex>[rr]^{\pi_1'} \ar@<-0.5 ex>[rr]_{\pi_2'}  \ar@{.>}[ru] && E' \ar[ru]}}
\end{equation}
that whenever the bottom equivalence relation is effective, the top one is effective as well. First of all recall that the bottom equivalence relation being effective means that $(\pi_1',\pi_2')$ is an effective equivalence relation in $\Ac$ and $\pi_1$ and $\pi_2$ are jointly monic. The cube being a discrete fibration in $\Reg(\Ac)$ means that the front square is a discrete fibration in $\Ac$ and, in the back square, $\overline{\pi}_2$ and $g$ are jointly monic. 

Since $(\pi_1',\pi_2')$ is the kernel pair of an effective descent morphism, by assumption, we have, by Lemma \ref{known}.2 that $(\overline{\pi}_1',\overline{\pi}_2')$   is an effective equivalence relation in $\Ac$, and it remains to be shown that $\overline{\pi}_1$ and $\overline{\pi}_2$ are jointly monic. However, this follows from the assumption that $\overline{\pi}_2$ and $g$ are jointly monic, as well as $\pi_1$ and $\pi_2$.
\end{proof}

Here's another surprise, or at least it was a surprise to us: the assumption that every regular epimorphism in $\Ac$ is an effective descent morphism is superfluous in the statement of the above theorem---it follows at once from the regularity of $\Reg(\Ac)$! To see this, we will need the following lemma. 

Recall that a diagram
\[
\xymatrix{
A''  \ar@<0.5 ex>[r]^{a_1}  \ar@<-0.5 ex>[r]_{a_2}  & A' \ar[r]^a& A}
\]
is called a \emph{fork} if $a\circ a_1=a\circ a_2$.

\begin{lemma}\label{George}
Let $\Ac$ be a regular category such that $\Reg(\Ac)$ is regular as well. 
Consider a morphism
\begin{equation}\label{forks}\vcenter{
\xymatrix{
E''  \ar@{}[rd]|<<{\pullback}  \ar@<0.5 ex>[r]^>>>>{e_1}  \ar@<-0.5 ex>[r]_>>>>{e_2} \ar[d] & E'  \ar@{}[rd]|<<{\pullback}    \ar[d] \ar[r]^{e} & E \ar[d]\\
B'' \ar@<0.5 ex>[r]^{b_1}  \ar@<-0.5 ex>[r]_{b_2}  & B' \ar[r]_b & B}}
\end{equation}
of forks in $\Ac$ such that the right hand square as well as the left hand squares are pullbacks. Assume, moreover, that the graph $(b_1,b_2)$ (hence also the graph $(e_1,e_2)$) is reflexive.  If $b$ is the coequaliser of $b_1$ and $b_2$, then $e$ is the coequaliser of $e_1$ and $e_2$.
\end{lemma}
\begin{proof}
The given diagram induces a commutative cube
\[
\xymatrix{
& E' \ar[rr]^{e}  \ar@{.>}[dd] && E \ar[dd] \\
E''  \ar[rr]^>>>>>>>>>{e_2}  \ar[dd] \ar[ru]^{e_1} && E' \ar[dd] \ar[ru]_{e}\\
& B' \ar@{.>}[rr]^<<<<<<<<<{}  && B\\
B'' \ar[rr]_{b_2}  \ar@{.>}[ru]^>>>>>>{b_1} && B' \ar[ru]_b}
\]
in which the bottom square is a pushout since $b$ is the coequaliser of $b_1$ and $b_2$, and because $b_1$ and $b_2$ have a common splitting. Since both front and back squares are pullbacks, and because each of the backward pointing morphisms is a regular epimorphism, the cube is a pullback in $\Reg(\Ac)$. As $\Reg(\Ac)$ was assumed to be regular, it follows that the top square is a pushout as well, so that $e$ is the coequaliser of  $e_1$ and $e_2$, as desired.
\end{proof}

Note that although we do not assume that arbitrary coequalisers exist in $\Ac$, coequalisers in the slice categories $(\Ac\downarrow B)$ are always coequalisers in $\Ac$:  indeed, the existence of binary products implies that the ``domain'' functors $(\Ac\downarrow B)\to \Ac$ have a right adjoint. The previous lemma is then easily seen to imply, for any morphism $p\colon E\to B$, that the  ``change of base''  functor $p^*\colon (\Ac\downarrow B)\to (\Ac\downarrow E)$ preserves coequalisers of reflexive graphs. (In fact, one can easily proof that the lemma is equivalent to this property---see also Remark \ref{admissibility}.) When $p$ is a regular epimorphism, then, moreover, $p^*$ reflects isomorphisms, as follows from Lemma \ref{known}.1. Since $p^*$ has a left adjoint $\Sigma_p$, we can then conclude, via the ``reflexive form'' of Beck's Monadicity Theorem, that $p^*$ is monadic.

Thus we have proved:
\begin{theorem}\label{new}
If $\Ac$ is a regular category such that $\Reg(\Ac)$ is regular as well, then every regular epimorphism in $\Ac$ is an effective descent morphism.
\end{theorem}

Combining Theorems \ref{maintheorem} and \ref{new} we obtain:
\begin{corollary}\label{corollary}
For a regular category $\Ac$ with pushouts of regular epimorphisms by regular epimorphisms, the following conditions are equivalent:
\begin{enumerate}
\item
Every regular epimorphism in $\Reg(\Ac)$ is an effective descent morphism;
\item
$\Reg(\Ac)$ is a regular category.
\end{enumerate}
\end{corollary}

\section{Examples}
The results of the previous section suggest to investigate which regular categories  $\Ac$ have the property that $\Reg(\Ac)$ is regular. We have the following examples:

\subsection{Exact Goursat categories}\label{ExGoursat}
Recall from \cite[Theorem 6.8]{Carboni-Kelly-Pedicchio} that a regular category is \emph{Goursat} if and only if for every (downward) morphism
\[
\xymatrix{
R' \ar@<0.5 ex>[r]^{\pi_1'}  \ar@<-0.5 ex>[r]_{\pi_2'} \ar[d]_r & E' \ar[d]^e\\
R \ar@<0.5 ex>[r]^{\pi_1}  \ar@<-0.5 ex>[r]_{\pi_2}  & E}
\]
of relations in $\Ac$ with $r$ and $e$ regular epimorphisms, $(\pi_1,\pi_2)$ is an equivalence relation as soon as so is $(\pi_1',\pi_2')$.  In universal algebra, Goursat varieties are usually called \emph{$3$-permutable varieties}, as they are characterised by the property $RSR=SRS$ for every two congruences $R$ and $S$ on any algebra $A$. Exact Goursat categories are easily seen to admit pushouts of regular epimorphisms by regular epimorphisms. Moreover, given the latter property, an exact category is Goursat if and only if it satisfies the following condition:   
\begin{itemize}
\item[($\mathsf{G}$)]
\emph{For any morphism
\begin{equation}\label{Goursat}\vcenter{
\xymatrix{
R' \ar@<0.5 ex>[r]  \ar@<-0.5 ex>[r] \ar[d]_r & E' \ar[d]^e \ar[r] & B' \ar[d]^b\\
R \ar@<0.5 ex>[r]  \ar@<-0.5 ex>[r]  & E \ar[r] & B}}
\end{equation}
of exact forks (which means that both rows consist of a regular epimorphism together with its kernel pair) with, moreover,  $e$ and $b$ regular epimorphisms, one has that the right hand square is a pushout if and only if $r$ is a regular epimorphism.}
\end{itemize}
Notice that  the ``if'' part is true in any category. In fact, in order to conclude that the right hand square in the diagram is a pushout it is sufficient that $r$ is an epimorphism.

Thus we have that a commutative square of regular epimorphisms in an exact Goursat category $\Ac$ is a regular epimorphism in $\Reg(\Ac)$ if and only if the induced morphism between the kernel pairs is a regular epimorphism. The regularity of $\Reg(\Ac)$ is now easily deduced from that of $\Ac$. Indeed, consider a pullback square
\begin{equation}\label{cube}\vcenter{
\xymatrix{
& P \ar[rr]^{}  \ar@{.>}[dd] && A \ar[dd] \\
P'  \ar[rr]^>>>>>>>>>{}  \ar[dd] \ar[ru] && A' \ar[dd] \ar[ru]\\
& E \ar@{.>}[rr]^<<<<<<<<<{}  && B\\
E' \ar[rr]^{}  \ar@{.>}[ru] && B' \ar[ru]}}
\end{equation}
in $\Reg(\Ac)$ and assume that the bottom morphism is a regular epimorphism: a pushout of regular epimorphisms in $\Ac$.  By taking kernel pairs we obtain a pullback square
\begin{equation}\label{inducedpb}\vcenter{
\xymatrix{
P'\times_P P' \ar[r] \ar[d] & A'\times_AA' \ar[d]\\
E'\times_EE' \ar[r] & B'\times_BB'}}
\end{equation}
in $\Ac$ in which, by condition ($\mathsf{G}$), the bottom morphism is a regular epimorphism. Hence, since $\Ac$ is a regular category, the top morphism is a regular epimorphism as well, and this implies, again by condition ($\mathsf{G}$), that the top side of the cube \eqref{cube} is a pushout of regular epimorphisms, as desired. In fact, it is clear that the following weaker condition is sufficient for $\Reg(\Ac)$ to be regular:
\begin{itemize}
\item[($\mathsf{G}^-$)]
\emph{For any morphism \eqref{Goursat} of exact forks with, moreover, $e$ and $b$ regular epimorphisms, the right hand square is a pushout if and only if $r$ is a pullback-stable epimorphism.}
\end{itemize}
Thus we conclude that if $\Ac$ is a regular category with pushouts of regular epimorphisms by regular epimorphisms, which satisfies the above condition ($\mathsf{G}^-$), then $\Reg(\Ac)$ is regular as well. Then, by Corollary \ref{corollary}, every regular epimorphism in $\Reg(\Ac)$ is an effective descent morphism. In particular, we find that this is the case if $\Ac$ is an exact Goursat category, a result which was already obtained in \cite{Janelidze-Sobral-Goursat}, via a different argument.

\subsection{Ideal determined categories}\label{ExIDT}
Recall from \cite{Janelidze-Marki-Tholen-Ursini} that a pointed finitely complete and finitely cocomplete regular category is \emph{ideal determined} if every regular epimorphism is normal (that is, it is the cokernel of its kernel) and if for any commutative square with $k$ and $e$ regular epimorphisms, and $\kappa$ and $\kappa'$  monomorphisms,
\[
\xymatrix{
K' \ar[r]^{\kappa'} \ar[d]_k & E'  \ar[d]^e\\
K \ar[r]_{\kappa} & E}
\]
if $\kappa'$ is normal (that is, it is the kernel of its cokernel) then $\kappa$ is normal as well. (Note that a pointed variety is ideal determined if it is ideal determined in the sense of \cite{Gumm-Ursini}.) This stability condition for normal monomorphisms is easily seen to be equivalent to the following ``normalised'' version of property ($\mathsf{G}$)---see also \cite{MM-NC}:
\begin{itemize}
\item[($\mathsf{Id}$)] \emph{For any commutative diagram 
\begin{equation}\label{ideal}\vcenter{
\xymatrix{
K' \ar[r] \ar[d]_k & E' \ar[r] \ar[d]^e & B'\ar[d]^b\\
K \ar[r] & E \ar[r] & B}}
\end{equation}
with (short) exact rows (which means that both rows consist of a regular epimorphism together with its kernel) with, moreover, $e$ and $b$ regular epimorphisms, 
the right hand square is a pushout if and only if $k$ is a regular epimorphism.}
\end{itemize}
As with condition ($\mathsf{G}$), the ``if'' part holds more generally: it is true in any pointed category in which every regular epimorphism is normal. Also, it is sufficient that $k$ is an epimorphism in order to conclude that the right hand side square in the diagram is a pushout.

Thus we have that a commutative square of regular epimorphisms in an ideal determined category $\Ac$ is a regular epimorphism in $\Reg(\Ac)$ if and only if the restriction to the kernels is a regular epimorphism. The arguments from the exact Goursat case are then easily adapted (simply take kernels instead of kernel pairs) in order to deduce the regularity of $\Reg(\Ac)$. And again, a weaker condition suffices:

\begin{itemize}
\item[($\mathsf{Id}^-$)] \emph{For any commutative diagram \eqref{ideal} with exact rows with, moreover, $e$ and $b$ regular epimorphisms, the right hand square is a pushout if and only if $k$ is a pullback-stable epimorphism.}
\end{itemize}
Thus we see that if $\Ac$ is a pointed regular category with pushouts of regular epimorphisms by regular epimorphisms, which satisfies the above condition ($\mathsf{Id}^-$), then $\Reg(\Ac)$ is regular as well. Then, by Corollary \ref{corollary}, every regular epimorphism in $\Reg(\Ac)$ is an effective descent morphism, and the same is true for the regular epimorphisms in $\Ac$, by Theorem \ref{new}. In particular, we find that these results hold for an ideal determined category $\Ac$. Note that for an ideal determined category $\Ac$ already the result that every regular epimorphism in $\Ac$ itself is an effective descent morphism is new, as far as we know. Note also that, in this case,  $\Reg(\Ac)$ is equivalent to the category of short exact sequences in $\Ac$, since every regular epimorphism is normal.

\subsection{Topological Mal'tsev algebras}
Recall that an algebraic theory $\mathbb{T}$ is a \emph{Mal'tsev} theory if it contains a ternary operation $p$ satisfying $p(x,y,y)=x$ and $p(x,x,y)=y$. The varieties $\Set^{\mathbb{T}}$ of $\mathbb{T}$-algebras for such theories $\mathbb{T}$ are exactly the congruence permutable ones---such that $RS=SR$ for any two congruences $R$ and $S$ on a same $\mathbb{T}$-algebra---and are often called \emph{Mal'tsev varieties}. Hence, every Mal'tsev variety $\Set^{\mathbb{T}}$ is Goursat. In particular, by Example \ref{ExGoursat},  $\Reg(\Set^{\mathbb{T}})$ is a regular category and every regular epimorphism in $\Reg(\Set^{\mathbb{T}})$ is an effective descent morphism, for any Mal'tsev theory $\mathbb{T}$.

Now, let us replace sets by topological spaces. More precisely, we consider categories $\Top^{\mathbb{T}}$ of topological $\mathbb{T}$-algebras for Mal'tsev theories $\mathbb{T}$. Contrary to the varieties case, the categories $\Top^{\mathbb{T}}$ are not Barr exact. However, they are well known to be regular, since regular epimorphisms are open surjections. Using the fact that the forgetful functor $\Top^{\mathbb{T}}\to \Set^{\mathbb{T}}$ preserves both limits and colimits, and reflects epimorphisms, it is then easily deduced from the Goursat condition ($\mathsf{G}$) for $\Set^{\mathbb{T}}$ that $\Top^{\mathbb{T}}$ satisfies the weaker condition ($\mathsf{G}^-$), for any Mal'tsev theory $\mathbb{T}$. Hence, by Example \ref{ExGoursat}, for a category of topological Mal'tsev algebras $\Top^{\mathbb{T}}$ the category $\Reg(\Top^{\mathbb{T}})$ is regular, and every regular epimorphism in it is an effective descent morphism.


%

\subsection{$n$-Fold regular epimorphisms}
For a category $\Ac$, let us put $\Reg^1\!(\Ac)=\Reg(\Ac)$ and define, inductively, for $n\geq 2$, the categories of \emph{$n$-fold regular epimorphisms} in $\Ac$ by $\Reg^n\!(\Ac)=\Reg(\Reg^{n-1}\!(\Ac))$. In any of the above considered examples, the category $\Reg^n\!(\Ac)$ is not only regular for $n=1$, but for \emph{any} $n\geq 1$. Indeed, if $\Ac$ is a regular category which admits pushouts of regular epimorphisms by regular epimorphisms and satisfies condition $(\mathsf{G}^-)$, then $\Reg(\Ac)$ has the same properties: the regularity follows from Example \ref{ExGoursat}, pushouts of regular epimorphisms by regular epimorphisms are degreewise pushouts in $\Ac$, and condition $(\mathsf{G}^-)$ follows from the corresponding condition on $\Ac$. Indeed, for the latter, recall that the ``if'' part of $(\mathsf{G}^-)$ is true in an arbitrary category, and notice that  a morphism
\begin{equation}\label{square}\vcenter{
\xymatrix{
E' \ar[d] \ar[r]^{p'} & B' \ar[d]\\
E \ar[r]_p \ar[r] & B}}
\end{equation}
in $\Reg(\Ac)$ is a pullback-stable epimorphism as soon as $p'$ is a pullback-stable epimorphism in $\Ac$.

Similarly, if $\Ac$ is a regular category in which every regular epimorphism is normal, which admits pushouts of regular epimorphisms by regular epimorphisms and satisfies condition $(\mathsf{Id}^-)$, then $\Reg(\Ac)$ has the same properties. It follows, in both cases, that the category $\Reg^n\!(\Ac)$ is regular for any $n\geq 1$.

In fact, we have the following, more general, property:
\begin{proposition}\label{goesup}
If $\Ac$ is a regular category such that also $\Reg(\Ac)$ is regular, then, for any $n\geq 2$, $\Reg^n\!(\Ac)$ is regular as well.
\end{proposition}
\begin{proof}
Of course, it suffices to prove that $\Reg^2\!(\Ac)$ is regular.

First of all note that $\Reg^2\!(\Ac)$ has coequalisers of effective equivalence relations, since $\Ac$, hence also $\Reg(\Ac)$, has pushouts of regular epimorphisms by regular epimorphisms. To see that $\Reg^2\!(\Ac)$ has pullback-stable regular epimorphisms, consider the functor  
\[
\dom\colon \Reg^2\!(\Ac)\to \Reg(\Ac)
\]
which sends a double regular epimorphism $(e\colon E'\to E)\to (b\colon B'\to B)$ to its domain $e$. Notice that $\dom$ preserves pullbacks. Moreover, $\dom$ preserves and reflects regular epimorphisms, as easily follows from the fact that a regular epimorphism in $\Reg^2\!(\Ac)$ is the same as a commutative cube in $\Ac$ of regular epimorphisms such that each of the sides is a pushout. Hence, the regularity of $\Reg^2\!(\Ac)$ follows from that of $\Reg(\Ac)$. 
\end{proof}

Combining the above proposition and Corollary \ref{corollary}, we find that if $\Ac$ is a regular category such that $\Reg(\Ac)$ is regular as well, then we have for any $n\geq 1$ that $\Reg^n\!(\Ac)$ is regular, and that every regular epimorphism in $\Reg^n\!(\Ac)$ is an effective descent morphism.

Remark that when $\Ac$ is an exact Mal'tsev category ($\Ac$ is exact and $RS=SR$ for any two equivalence relations $R$ and $S$ on a same object of $\Ac$) then by a result in \cite{Carboni-Kelly-Pedicchio} a pushout of regular epimorphisms is the same as a \emph{double extension} (a notion from ``higher dimensional'' Galois theory---see, for instance, \cite{Janelidze:Double,EGVdL}): a commutative square \eqref{square} of regular epimorphisms, such that the factorisation  $E'\to E\times_BB'$ to the pullback is a regular epimorphism as well. 

Notice that Proposition \ref{goesup} together with Theorem \ref{new}  provide an alternative proof for Corollary \ref{corollary}.



\section{Remarks}
\subsection{} The relations between the various conditions considered above---conditions ($\mathsf{G}$), ($\mathsf{G}^-$), ($\mathsf{Id}$), ($\mathsf{Id}^-$) and the condition that $\Reg(\Ac)$ is a regular category---are still to be better understood. However, we do know the following: 

-Any of these conditions holds in any semi-abelian category \cite{Janelidze-Marki-Tholen}. Hence, we obtain as examples any variety of groups, rings, Lie algebras, and, more generally, any variety of $\Omega$-groups; the variety of loops; the variety of Heyting semi-lattices;  any abelian category; \dots.   

-There exist ideal determined varieties that are not $3$-permutable (see \cite{Barbour-Raftery}), hence ($\mathsf{Id}$) does not imply ($\mathsf{G}$). In particular, this means that $3$-permutability is not a necessary condition on a variety $\Ac$ in order for every regular epimorphism in  $\Reg(\Ac)$ to be an effective descent morphism. This answers a question posed in \cite{Janelidze-Sobral-Goursat}. 

-The examples of topological Mal'tsev algebras and of $n$-fold regular epimorphisms show that the weaker conditions ($\mathsf{G}^-$) and ($\mathsf{Id}^-$) are strictly weaker than  ($\mathsf{G}$) and ($\mathsf{Id}$).

Let us also remark that we do not know, at present, any example of a regular category whose category of regular epimorphisms is regular, which does not satisfy either condition ($\mathsf{G}^-$) or ($\mathsf{Id}^-$).

%
%

\subsection{}\label{admissibility}
As mentioned earlier,  Lemma \ref{George} precisely says the following, for $\Ac$ a regular category such that $\Reg(\Ac)$ is regular as well:

\emph{For any morphism $p\colon E\to B$ in $\Ac$,  the  ``change of base''  functor $p^*\colon (\Ac\downarrow B)\to (\Ac\downarrow E)$ preserves coequalisers of reflexive graphs.}

Here's another way of reformulating this same property. Recall from \cite{CHK} that a reflector $I\colon \Ac\to\Xc$ into a full subcategory $\Xc$ of a category $\Ac$ is \emph{semi-left exact} if it preserves any pullback square
\[
\xymatrix{
P \ar@{}[rd]|<<{\pullback}\ar[r] \ar[d] & X \ar[d]\\
B \ar[r]_-{\eta_B} & I(B)}
\]
of a unit $\eta_B$ along a morphism in $\Xc$. Clearly, this means that, for every such pullback square, the morphism $P\to X$ coincides, up to isomorphism, with the unit $\eta_P\colon P\to I(P)$. Note that semi-left exactness is the same as \emph{admissibility} in the sense of categorical Galois theory (see \cite{CJKP}).

Now, let $\Ac$ be a regular category such that $\Reg(\Ac)$ is regular as well, and $\RG(\Ac)$ the category of reflexive graphs in $\Ac$.
If we assume, moreover, that $\Ac$ admits coequalisers of reflexive graphs, then Lemma \ref{George} can also be reformulated as follows:

\emph{The functor $\pi_0\colon \RG(\Ac)\to \Ac$, which sends a reflexive graph to its coequaliser, is semi-left exact.}

\subsection{}
For an object $B$ of a category $\Ac$, write $\Pt(B)$ for the category defined as follows: an object of $\Pt(B)$ is a triple $(A,f,s)$, where $A$ is an object of $\Ac$ and $f\colon A\to B$ and $s\colon B\to A$ are morphisms in $\Ac$ such that $f\circ s=1_B$; a morphism $(A,f,s)\to (C,g,t)$ in $\Pt(B)$ is a morphism $h\colon A\to C$ in $\Ac$ such that $g\circ h=f$ and $h\circ s=t$. 

When $\Ac$ has pullbacks of split epimorphisms, any morphism $p\colon E\to B$ in $\Ac$ induces a ``change of base'' functor $p^*\colon \Pt(B)\to \Pt(E)$ given by pulling back along $p$. If, moreover,  $\Ac$ admits pushouts of split monomorphisms, then any such $p^*$ has a left adjoint (see \cite{Bourn-Janelidze:Semidirect}). Recall from \cite{Bourn1991} that $\Ac$ is called \emph{protomodular} if $p^*$ reflects isomorphisms for every morphism $p$.

Now, let $\Ac$ be a regular category such that $\Reg(\Ac)$ is regular as well, and assume that $\Ac$ admits coequalisers of reflexive graphs. Then,  for any object $B$ of $\Ac$, coequalisers of reflexive graphs in $\Pt(B)$ are necessarily coequalisers in $\Ac$,  and we can use Lemma \ref{George} to prove that they are preserved by $p^*\colon \Pt(B)\to \Pt(E)$, for every morphism $p\colon E\to B$ in $\Ac$. Hence, using the ``reflexive form'' of Beck's Monadicity Theorem we find, for any regular protomodular  category $\Ac$ with coequalisers of reflexive graphs and pushouts of split monomorphisms such that  $\Reg(\Ac)$ is regular as well,  that the functor $p^*\colon \Pt(B)\to \Pt(E)$ is monadic for any morphism $p\colon E\to B$ in $\Ac$. This means, according to  \cite{Bourn-Janelidze:Semidirect},  that $\Ac$  is a \emph{category with semidirect products}.

\subsection{}
Let $\Ac$ be a regular category and $\Mon(\Ac)$ the full subcategory of the arrow category $\Ac^2$ with as objects all monomorphisms in $\Ac$.  Like $\Ac$, the category $\Mon(\Ac)$ is finitely complete:  limits in $\Mon(\Ac)$ are degreewise limits in $\Ac$. Colimits, if they exist, are given by the mono part of the regular epi-mono factorisation of the degreewise colimit. A regular epimorphism in $\Mon(\Ac)$ is the same as a degreewise regular epimorphism in $\Ac$. In particular, we have that  $\Mon(\Ac)$ is a regular category. 

Now let  $\cod\colon \Mon(\Ac)\to\Ac$ be the functor which sends a monomorphism $a\colon A'\to A$ to its codomain $A$. Then $\cod$ preserves pullbacks, pushouts and regular epi-mono factorisations, and reflects pushouts of regular epimorphisms in the following sense: any commutative square of regular epimorphisms in $\Mon(\Ac)$ that is sent to a pushout in $\Ac$ is necessarily a pushout itself. It follows that whenever $\Reg(\Ac)$ is regular, $\Reg(\Mon(\Ac))$ is regular as well. Hence, categories of monomorphisms provide another class of examples of regular categories whose category of regular epimorphisms is regular as well. In particular, we find that $\Reg(\Mon(\Ac))$ is regular for any semi-abelian category $\Ac$. In this case $\Mon(\Ac)$ is also protomodular and finitely complete, so that by the previous remark $\Mon(\Ac)$ admits semi-direct products, a result which was already obtained in \cite{Ferreira-Sobral}.

\subsection{} Let $\Ac$ be a regular category. By an \emph{extension} we simply mean a regular epimorphism in $\Ac$; a \emph{double extension} is, as recalled above, a commutative square of extensions such that the induced factorisation to the pullback is an extension as well. Let us write $\Ext(\Ac)$ and $\Ext^2\!(\Ac)$ for the categories of extensions and of double extensions, respectively (where we have that $\Ext(\Ac)=\Reg(\Ac)$, of course). One can then, inductively, define \emph{$n$-fold extensions} for $n\geq 3$, as those commutative squares of ($n-1$)-fold extensions such that the induced factorisation to the pullback (in the category $\Ext^{n-2}\!(\Ac)$) is an  ($n-1$)-fold extension.  It was noted in \cite{EGVdL} that every $n$-fold extension (for $n\geq 2$) is an effective descent morphism in $\Ext^{n-1}\!(\Ac)$, basically because pullbacks along $n$-fold extensions are computed degreewise in the exact category $\Ac$.

As mentioned earlier, if $\Ac$ is an exact Mal'tsev category, then a regular epimorphism in $\Ext(\Ac)=\Reg(\Ac)$ is the same as a double extension, by a result in \cite{Carboni-Kelly-Pedicchio}. In fact, it is shown in \cite{Carboni-Kelly-Pedicchio} that this property characterises the exact Mal'tsev categories $\Ac$ among the regular ones. Hence, we have, in $\Ac$, that the effective descent morphisms are exactly the extensions, since $\Ac$ is exact, and, in $\Ext(\Ac)$, exactly the double extensions, by Theorem \ref{maintheorem} and the result in  \cite{Carboni-Kelly-Pedicchio}. This might lead one to expect that the effective descent morphisms in $\Ext^2(\Ac)$ are exactly the three-fold extensions. However, this turns out not to be the case, in general, and happens only when the category $\Ac$ is \emph{arithmetical} in the sense of \cite{Pedicchio2} (for instance, when $\Ac$ is the variety of Boolean rings, or of Heyting algebras). Indeed, by Theorem \ref{maintheorem}, the effective descent morphisms in $\Ext^2(\Ac)$ are exactly the pushout squares of regular epimorphisms in $\Ext(\Ac)$, so that, by the above mentioned result in \cite{Carboni-Kelly-Pedicchio}, we would have that $\Ext(\Ac)$ is exact Mal'tsev which, by a result in \cite{Bourn2001b} is the case if and only if $\Ac$ is an arithmetical category.
\\ \\
\emph{Acknowledgement.} Many thanks to George Janelidze for pointing out that in the assumptions of Theorem \ref{maintheorem} the ``almost exactness'' of the category $\Ac$ is superfluous.

%
%

\providecommand{\bysame}{\leavevmode\hbox to3em{\hrulefill}\thinspace}
\providecommand{\MR}{\relax\ifhmode\unskip\space\fi MR }
\providecommand{\MRhref}[2]{%
  \href{http://www.ams.org/mathscinet-getitem?mr=#1}{#2}
}
\providecommand{\href}[2]{#2}


\begin{thebibliography}{10}



\bibitem{Barbour-Raftery}
G.~D. Barbour and J.~G. Raftery, \emph{Ideal determined varieties have
  unbounded degrees of permutability}, Quaest.~Math. \textbf{20} (1997),
  563--568.

\bibitem{Benabou-Roubaud}
J.~B{\'e}nabou and J.~Roubaud, \emph{Monades et descente}, C.R. Acad. Sc.
  \textbf{270} (1970), 96--98.

\bibitem{Bourn1991}
D.~Bourn, \emph{Normalization equivalence, kernel equivalence and affine
  categories}, Category {T}heory, {P}roceedings {C}omo 1990 (A.~Carboni, M.~C.
  Pedicchio, and G.~Rosolini, eds.), Lecture Notes in Math., vol. 1488,
  Springer, 1991, pp.~43--62.

\bibitem{Bourn2001b}
D.~Bourn, \emph{A categorical genealogy for the congruence distributive
  property}, Theory Appl. Categ. \textbf{8} (2001), no.~14, 391--407.

\bibitem{Bourn-Janelidze:Semidirect}
D.~Bourn and G.~Janelidze, \emph{Protomodularity, descent, and semidirect
  products}, Theory Appl. Categ. \textbf{4} (1998), no.~2, 37--46.

\bibitem{Carboni-Kelly-Pedicchio}
A.~Carboni, G.~M. Kelly, and M.~C. Pedicchio, \emph{Some remarks on {M}altsev
  and {G}oursat categories}, Appl. Categ. Structures \textbf{1} (1993),
  385--421.

\bibitem{CHK} C.~Cassidy, M.~H\'ebert and G.~M. Kelly, \emph{Reflective subcategories, localizations and factorizations systems}, J. Austral. Math. Soc. \textbf{38} (1985), 287-329.

\bibitem{CJKP} A. Carboni, G. Janelidze, G. M. Kelly, and R. Par\'e, \emph{On localization and stabilization of factorization systems}, Appl. Categ. Structures 5, 1-58 (1997)

\bibitem{EGVdL}
T.~Everaert, M.~Gran, and T.~Van~der Linden, \emph{Higher {H}opf formulae for
  homology via {G}alois {T}heory}, Adv.~Math. \textbf{217} (2008), 2231--2267.


\bibitem{Gumm-Ursini}
H.~P. Gumm and A.~Ursini, \emph{Ideals in universal algebra}, Algebra
  Universalis \textbf{19} (1984), 45--54.

\bibitem{Janelidze:Double}
G.~Janelidze, \emph{What is a double central extension? ({T}he question was
  asked by {R}onald {B}rown)}, Cah. Topol. G{\'e}om. Differ. Cat{\'e}g.
  \textbf{XXXII} (1991), no.~3, 191--201.

\bibitem{Janelidze-Marki-Tholen}
G.~Janelidze, L.~M{\'a}rki, and W.~Tholen, \emph{Semi-abelian categories},
  J.~Pure Appl. Algebra \textbf{168} (2002), 367--386.

\bibitem{Janelidze-Marki-Tholen-Ursini}
G.~Janelidze, L.~M{\'a}rki, W.~Tholen, and A.~Ursini, \emph{Ideal determined
  categories}, Cah. Topol. G\a'eom. Diff\a'er. Cat\a'eg. \textbf{LI-2} (2010),
  115--125.

\bibitem{Janelidze-Sobral-Goursat}
G.~Janelidze and M.~Sobral, \emph{Descent for regular epimorphisms in {B}arr
  exact {G}oursat categories}, Appl.\ Categ.\ Structures, published online 13th
  May 2009.

\bibitem{Janelidze-Sobral-Tholen}
G.~Janelidze, M.~Sobral, and W.~Tholen, \emph{Beyond {B}arr exactness:
  Effective descent morphisms}, {C}ategorical Foundations: Special Topics in
  Order, Topology, Algebra and Sheaf Theory (M.~C. Pedicchio and W.~Tholen,
  eds.), Encyclopedia of Math. Appl., vol.~97, Cambridge University Press,
  2004, pp.~359--405.

\bibitem{Janelidze-Tholen2}
G.~Janelidze and W.~Tholen, \emph{Facets of descent {II}}, Appl. Categ.
  Structures \textbf{5} (1997), 229--248.

\bibitem{MM-NC}
S.~Mantovani and G.~Metere, \emph{Normalities and commutators}, J.~Algebra
  \textbf{324} (2010), no.~9, 2568--2588.

\bibitem{Ferreira-Sobral}
N.~Martins-Ferreira and M.~Sobral, \emph{On categories with semidirect products} (2010, preprint).



\bibitem{Pedicchio2}
M.~C. Pedicchio, \emph{Arithmetical categories and commutator theory}, Appl.
  Categ. Structures \textbf{4} (1996), no.~2--3, 297--305.

\end{thebibliography}
\end{document}